\documentclass[11pt]{amsart}
\include{epic.sty}
\usepackage{latexsym,amssymb}

\usepackage{amsmath}
\usepackage{amsthm}
\usepackage{mathdots}
\usepackage{color}
\usepackage{multirow}
\usepackage{bm}
\usepackage{hyperref}
\usepackage[margin=1.25in]{geometry}

\usepackage[latin1]{inputenc}

\newtheorem{proposition}{Proposition}
\newtheorem{theorem}{Theorem}

\newcommand{\ZZ}{\mathbb{Z}}

\newtheorem{lemma}{Lemma}

\newcommand{\End}{\text{End}}

\newcommand{\id}{\mathbf{1}}
\newcommand{\QQ}{\mathbb{Q}}

\newcommand{\Sch}{\mathcal{S}}
\newcommand{\db}{{}^{\mathfrak{b}}d}

\newcommand{\superpoly}{\QQ[\mathbf{x},\mathbf{dx}]}
\newcommand{\superpolydenoms}{\QQ[\mathbf{x},\mathbf{dx},\bm{\alpha}]}
\newcommand{\extpoly}{\QQ[\mathbf{x},\bm{\omega}]}
\newcommand{\p}{\mathbf{p}}
\newcommand{\x}{\mathbf{x}}
\newcommand{\f}{\mathbf{f}}
\newcommand{\df}{\mathbf{df}}
\newcommand{\PP}{\mathsf{P}}
\newcommand{\J}{{}^{\mathfrak{b}}\mathbf{\mathsf{J}}_p^f }

\newcommand{\nhb}{{}^{\mathfrak{b}}\mathsf{NH}^{ext}}
\newcommand{\schub}{{}^{\mathfrak{b}}\mathfrak{s}}

\begin{document}
\title{Extended nilHecke algebras and symmetric functions in type B}
\author{Michael Reeks}
\address{Department of Mathematics\\ University of Virginia \\ Charlottesville, VA}
\email{mar3nf@virginia.edu}

\begin{abstract} We formulate a type B extended nilHecke algebra, following the type A construction of Naisse and Vaz. We describe an action of this algebra on extended polynomials and describe some results on the structure on the extended symmetric polynomials. Finally, following Appel, Egilmez, Hogancamp, and Lauda, we prove a result analogous to a classical theorem of Solomon connecting the extended symmetric polynomial ring to a ring of usual symmetric polynomials and their differentials. \end{abstract}
\maketitle

\section{Introduction}
Affine Hecke algebras have many diverse applications in representation theory, and their nil versions are fundamental in the study of categorification and higher representation theory. Type A nilHecke algebras appear as the quiver Hecke algebras associated to a single vertex, which were shown by Lauda (\cite{Lau}) to categorify the negative half of the quantum group $\mathcal{U}_q(\mathfrak{sl}_2)$. The nilHecke algebras  form an essential building block of the Khovanov--Lauda--Rouquier categorification of the quantum group associated to an arbitrary Kac--Moody algebra. The nilHecke algebras in classical types (in particular, in type B) also play a fundamental role in Schubert calculus (cf. \cite{Ku,FK}). 

An extended form of the nilHecke algebra in type A, $\mathsf{NH}_n^{ext}$ was constructed by Naisse and Vaz in \cite{NV1}. The algebra $\mathsf{NH}_n^{ext}$ has an additional set of generators $\omega_i$, which anticommute and square to 0. This algebra was used to construct the first categorifications of Verma modules. It has a faithful representation on a ring of extended polynomials, which are the tensor product of a polynomial ring and an exterior algebra. In \cite{AEHL}, Appel, Egilmez, Hogancamp, and Lauda study the combinatorics of this algebra and its polynomial representation, and provide additional descriptions of the ring of extended symmetric polynomials. They further show that $\mathsf{NH}_n^{ext}$ is a matrix algebra over these invariants (this was also shown independently in \cite{NV2}). They also prove an extended analogue of a theorem of Solomon (\cite{Sol}), which relates the extended symmetric polynomials to the invariants of $\QQ[\mathbf{x}] \otimes \bigwedge[\mathbf{dx}].$

In this note, we formulate a type B version of the extended nilHecke algebra, $\nhb$, and of its extended polynomial representation. A new feature here is the action of the simple reflections $s_i$ and the Demazure operators $\partial_i$ on the odd polynomial generators $\omega_i$: we set  $$s_i(\omega_i) = \omega_i + (x_i^2 - x_{i+1}^2)\omega_{i+1}$$ and $$\partial_i(\omega_j) = -\delta_{ij}(x_i+x_{i+1})\omega_{i+1}$$ for each $i$. This shifted action is more natural from the viewpoint of type B symmetric polynomials, which are generated by usual symmetric polynomials in the variables $x_i^2$, and facilitates the connection to Solomon's theorem. We then establish several basis and dimension results about the type B extended symmetric polynomials following \cite{AEHL}, demonstrate that ${}^{\mathfrak{b}}\mathsf{NH}_n^{ext}$ is a matrix algebra over these invariants. Many of these results and their proofs are parallel to the type A case discussed in \cite{AEHL} and \cite{NV2}.

We nonetheless note that this construction is a nontrivial extension of the type A situation, and that the type D extended nilHecke algebra seems to be even more difficult.  While it is relatively easy to find an action of the type D Demazure operators on the algebra of extended polynomials, most such actions do not preserve the expected relationships between the degrees of the $\omega_i$ and those of the type D symmetric polynomials as would be required to construct the matrix for the extended Solomon's theorem. 

The paper is structured as follows.  In Section 2, we describe the extended polynomial representation of the Weyl group of type $B_n$, and use it to define an extended polynomial representation of $\nhb_n$. In Section 3, we describe the extended symmetric polynomial ring. We describe a basis of extended Schur polynomials and show that $\nhb_n$ is a matrix ring over it. Finally, in Section 4, we formulate an extended type B analogue of Solomon's therorem, linking the extended type B symmetric polynomials to the type B invariants of $\QQ[\mathbf{x}] \otimes \bigwedge[\mathbf{dx}].$ 

\subsection{Acknowledgements}

The author thanks Weiqiang Wang for many helpful discussions concerning the paper, and Matthew Hogancamp for his collaboration in constructing the differentials in Section 2.3.


\section{The type B extended nilHecke algebra}

We review the definition of the type B Weyl group, and define an action of this group on extended polynomials. Using this action, we define a type B extended nilHecke algebra, derive a presentation with generators and relations, and prove that it has a PBW-type basis. 

\subsection{An action of $W_{B_n}$ on extended polynomials}
The Weyl group of type $B_n$, $W_{B_n}$ is generated by $s_1, \ldots, s_n$ with relations such that $s_1, \ldots, s_{n-1}$ generate a subgroup isomorphic to the symmetric group $S_n$, and additional relations $$s_n s_{n-1} s_ns_{n-1} = s_{n-1} s_n s_{n-1} s_n,$$ $$s_n s_i = s_i s_n \qquad (1\leq i \leq n-2),$$ $$s_n^2=1.$$ Define the extended polynomial ring $$\mathsf{P}_n^{ext} = \QQ[x_1,\ldots,x_n] \otimes \bigwedge[\omega_1, \ldots, \omega_n].$$ This ring is graded with $\deg(x_i) = 1$ and $\deg(\omega_i) = -2i$. 

Define an action of $W_{B_n}$ on the ring $\mathsf{P}_n^{ext}$ of extended polynomials by setting  $$s_i(x_j) = x_{s_i(j)}, \quad s_n(x_i) = x_i \qquad (1\leq i \leq n-1),$$ $$s_n(x_n) = -x_n,$$ $$s_i(\omega_j) = \omega_j + \delta_{ij}(x_i^2 - x_{i+1}^2) \omega_{i+1} \qquad (1\leq i \leq n-1),$$ $$s_n(\omega_n) = \omega_n,$$ and letting the $s_i$ act as automorphisms. 
\begin{lemma} The above defines an action of $W_{B_n}$ on $\mathsf{P}_n^{ext}$. \end{lemma}
\begin{proof} First, we check that the type A relations are satisfied. Note that, for any $i<n$, we have $$s_i^2(\omega_i) = \omega_i + (x_i^2 - x_{i+1}^2)\omega_{i+1} + (x_{i+1}^2- x_i^2)\omega_{i+1}$$ $$= \omega_i.$$ Next, note that, for any $i<n-1$, $$s_i s_{i+1} s_i(\omega_i) = \omega_i + (x_i^2 - x_{i+1}^2) \omega_{i+1} + (x_{i+1}^2 - x_{i+2}^2)\omega_{i+1} + ( x_{i+1}^2 - x_{i+2}^2)(x_i^2-x_{i+2}^2) \omega_{i+2} = s_{i+1} s_i(\omega_i).$$ Clearly $s_{i+1} s_i s_{i+1}(\omega_i) = s_{i+1} s_i(\omega_i)$, so this braid relation is satisfied. Finally, we have $$s_is_{i+1}s_i(\omega_{i+1}) = s_i s_{i+1}(\omega_{i+1}) = \omega_{i+1}+(x_i^2 - x_{i+2}^2)\omega_{i+2}.$$ On the other hand, $$s_{i+1} s_i s_{i+1}(\omega_{i+1} ) = s_{i+1}( \omega_{i+1}+(x_i^2 - x_{i+2}^2)\omega_{i+2})$$ $$= \omega_{i+1}+(x_i^2 - x_{i+2}^2)\omega_{i+2}.$$ Hence, all type A relations are satisfied.

 Now it remains to check that the action of $W_{B_n}$ on $\omega_n$ satisfies the type B relations. We compute: $$s_{n-1} s_n s_{n-1} s_n (\omega_n) = s_n s_{n-1} s_n s_{n-1} (\omega_n) = \omega_n.$$ Next, note that $$
s_{n-1} s_n s_{n-1} s_n (\omega_{n-1}) = \omega_{n-1} + 2x_{n-1}^2 \omega_n.$$ Furthermore, since both $\omega_{n-1}$ and $ \omega_{n-1} + 2x_{n-1}^2 \omega_n$ are symmetric with respect to $s_n$, it follows that the latter is equal to $s_n s_{n-1} s_n s_{n-1} (\omega_{n-1})$. It's easy to see that the action of $s_n$ commutes with that of $s_i$ for any $1\leq i \leq n$. \end{proof}

This action induces an action of type B Demazure operators. Define the Demazure operator $\partial_i$ by $$\partial_i = \frac{\id - s_i}{x_i - s_i(x_i)}.$$ In particular, note that $$\partial_n = \frac{\id - s_n}{2 x_n}.$$ It's easy to check the following actions: for $1\leq i \leq n-1$ and $1\leq j \leq n$ we have $$\partial_i(x_j) = \left\{\begin{array}{lr} 1 & \ \text{if }i=j \\ -1 &\ \text{if }j=i+1 \\  0&\ \text{else,}\end{array}\right. \qquad \partial_i(\omega_j) = -\delta_{ij} (x_i+x_{i+1})w_{i+1}.$$ Finally, we have $$\partial_n(x_j) = \delta_{jn} \qquad \partial_n(\omega_j) = 0$$ for all $1\leq j \leq n$. Extend this action to an arbitrary polynomial by the Leibniz rule $$\partial_i(fg) = \partial_i(f) g + s_i(f) \partial_i(g).$$ \begin{lemma} The above defines an action of the type B Demazure operators on $\mathsf{P}_n^{ext}$. \end{lemma}

\begin{proof} We again first check the type A relations. Note that for any $i$, $$\partial_i^2(\omega_i) = \partial_i(-(x_i + x_{i+1})\omega_{i+1}) = 0.$$ Now, for any $i<n-1$, we have \begin{align*} \partial_i \partial_{i+1} \partial_i(\omega_i) &= \partial_i\partial_{i+1}(-(x_i + x_{i+1})\omega_{i+1}) \\&= \partial_1(\omega_2 + (x_1 + x_3)(x_2 + x_3) \omega_3) \\ &= (x_3 - x_3)\omega_3 \\&=0. \end{align*} Clearly $\partial_2 \partial_1 \partial_2(\omega_1) = 0$. For the other braid relation, note that $$\partial_2 \partial_1 \partial_2(\omega_2) = \partial_2 (-\omega_3) = 0,$$ whereas again $\partial_1\partial_2\partial_1(\omega_2)=0$ immediately. Hence, the type A relations are satisfied.

It remains to check the extra type B relations on generators: first note that $$\partial_{n-1} \partial_n \partial_{n-1} \partial_n (\omega_n) = 0 = \partial_n \partial_{n-1} \partial_n \partial_{n-1}(\omega_n).$$ Also, we have $$\partial_{n-1} \partial_n \partial_{n-1} \partial_n (\omega_{n-1}) = 0,$$ whereas $$\partial_n \partial_{n-1} \partial_n \partial_{n-1} (\omega_{n-1}) = -\partial_n \partial_{n-1} \partial_n(\omega_n) = 0.$$ It is clear that the action of $\partial_n$ commutes with that of $\partial_i$, $1\leq i \leq n$. 
\end{proof}

\subsection{Extended nilHecke algebra}

Define the extended nilHecke algebra of type $B$, $\nhb$, to be the $\QQ$-superalgebra generated by multiplication in $\mathsf{P}_n^{ext}$, together with the action of the Demazure operators on extended polynomials. The generators $\omega_i$ are odd, and all other generators are even. There is an additional $\ZZ$-grading, with $\deg(x_i) = 1$, $\deg(\partial_i)=-1$, and $\deg(\omega_i) = -2i.$ We give a presentation of this algebra in terms of generators and relations and prove a PBW-type theorem.

The degenerate affine Hecke algebra in type B (and in all finite types) was described in \cite{Lu}. It's nil version, the type B nilHecke algebra ${}^{\mathfrak{b}}\mathsf{NH}_n$, is the $\QQ$-algebra generated by $\partial_1, \ldots, \partial_n$ and $x_1, \ldots, x_n$, with relations $x_ix_j = x_jx_i$ and, for $i,j<n$, $$ \partial_i^2 = 0, \qquad \partial_i\partial_j = \partial_j \partial_i\ \text{if } |i-j|>1\qquad  \ 1\leq i \leq n,$$ $$ \partial_i\partial_{i+1}\partial_i = \partial_{i+1}\partial_i \partial_{i+1}, \ 1\leq i <n-1, $$ $$\partial_{n-1}\partial_n \partial_{n-1} \partial_n = \partial_n\partial_{n-1} \partial_n \partial_{n-1} $$ $$\partial_i x_j = x_j \partial_i \ \text{if }|i-j|>1, \qquad \partial_i x_i - x_{i+1} \partial_i = 1 , \qquad \partial_i x_{i+1} - x_i \partial_i = -1 \ 1\leq i <n,$$ $$\partial_n x_n  + x_n \partial_n = 1, \qquad x_n \partial_n + \partial_n x_n = -1.$$ This is a $\ZZ$-graded algebra with $\deg(x_i) = 1$ and $\deg(\partial_i) = -1$.

\begin{proposition} There is an isomorphism of $\QQ$-superalgebras $$\nhb_n \cong {}^{\mathfrak{b}}\mathsf{NH}_n \rtimes \bigwedge[\omega_1, \ldots, \omega_n],$$ where the generators on the right-hand side satisfy the relations $x_i \omega_j = \omega_j x_i $ for all $i$ and $j$, as well as $$\partial_i \omega_j = \omega_j \partial_i \quad i\not=j,$$ $$\partial_i(\omega_i - x_{i+1}^2 \omega_{i+1}) = (\omega_i - x_{i+1}^2 \omega_{i+1}) \partial_i, \ i<n,$$ $$\partial_n\omega_i = \omega_i \partial_n\quad\ 1\leq i \leq n.$$ \end{proposition} \begin{proof} We only need to prove the additional relations.  The first two and the last relations are clear. Finally, note that $$\partial_i(\omega_i) = -(x_i + x_{i+1})\omega_{i+1} = \partial_i (x_{i+1}^2 \omega_{i+1}),$$ which implies the remaining relation. \end{proof} 

Using these relations and the action on $\mathsf{P}_n^{ext}$, we obtain the following PBW basis for $\nhb$. For $w\in W_{B_n}$ with reduced expression $w= s_{i_1} \ldots s_{i_k}$, set $\partial_w = \partial_{i_1} \ldots \partial_{i_k}.$ 

\begin{theorem}\label{pbwbasis} The superalgebra $\nhb$ has a $\QQ$-basis given by $$\{x_i^{k_1} \ldots x_n^{k_n} \omega_i^{\epsilon_1} \ldots \omega_n^{\epsilon_n} \partial_w | k_i\in \mathbb{N},\ \epsilon_i \in \{0,1\},\ w\in W_{B_n}\}.$$ \end{theorem} \begin{proof} Using the relations, we can put any given monomial in this form, so this set spans $\nhb$. But $\nhb$ acts faithfully on $\mathsf{P}_n^{ext}$, which implies that this set is linearly independent and thus a basis. \end{proof} 

Note that we could obtain alternative bases for $\nhb$ by exchanging the positions of the $x_i$'s, $\omega_i$'s, and $\partial_w$ in the basis given by Theorem \ref{pbwbasis}.

\subsection{Differentials}

This section is joint with Matthew Hogancamp.

The algebra $\nhb$ has a natural differential graded (DG) structure. Recall that a DG-algebra is a $\ZZ$-graded unital algebra $A$ with a degree $1$ differential $d:A\rightarrow A$ satisfying $d^2=0,$ $d(1)=0$ and, for $a,b\in A$, \begin{equation}\label{dgstructure} d(ab ) = d(a)b + (-1)^{|a||b|}a d(b).\end{equation}

For each $N>0$, define a grading on $\nhb_k$ by setting $\deg_N(x_i) = 1$, $\deg_N(\partial_i) = -1$ and $\deg_N(\omega_i) = 2(N-i)+1$. Define a map ${}^{\mathfrak{b}}d_N:\nhb_k \rightarrow \nhb_k$ by setting $$\db_N(x_i) = \db_N(\partial_i) = 0 \qquad \db_N(\omega_i) = (-1)^i\  {}^{\mathfrak{b}}h_{N-i+1}(1,i)$$ where ${}^{\mathfrak{b}}h_{n}(i,j)$ is the complete homogeneous symmetric polynomial in the variables $x_i^2, x_{i+1}^2, \ldots, x_j^2$ (the type B homogeneous symmetric polynomial). 

\begin{proposition} For each $N$, the map $\db_N$ defines a differential making $\nhb_k$ into a DG-algebra with grading given by $\deg_N$. \end{proposition}

\begin{proof}
Clearly $\db_N^2=0,$ $\db_N(1)=0$. Since $\deg_N({}^{\mathfrak{b}}h_{N-i+1}(1,i)) = 2(N-i+1) = \deg_N(\omega_i)+1$, we have $\deg(d_N) = 1$. 

Note that we have $$d_N(\omega_1) = -x_1^{2N}.$$ Since $$ \partial_i(\db_N(\omega_i)) =\db_N(\partial_i(\omega_i)) = -(x_i+x_{i+1})\db_N(\omega_{i+1}),$$ it suffices to show that \begin{equation}\label{homgenpolys}\partial_i({}^{\mathfrak{b}}h_{\ell}(1,i)) = -(x_i+x_{i+1}){}^{\mathfrak{b}}h_{\ell-1}(1,i+1).\end{equation} We apply a generating function argument. We have \begin{align*} \partial_i\left(\prod_{j=1}^i \frac{1}{1-tx_j^2}\right) &= \left(\prod_{j=1}^{i-1}\frac{1}{1-tx_j^2}\right) \left(\frac{1}{1-tx_i^2} - \frac{1}{1-tx_{i+1}^2}\right) \frac{1}{x_i-x_{i+1}} \\ &=  \left(\prod_{j=1}^{i-1}\frac{1}{1-tx_j^2}\right) \frac{tx_i^2 - tx_{i+1}^2}{(1-tx_i^2)(1-tx_{i+1}^2)(x_i-x_{i+1})} \\ &=  \left(\prod_{j=1}^{i-1}\frac{1}{1-tx_j^2}\right) \frac{t (x_i+x_{i+1})}{(1-tx_i^2)(1-tx_{i+1}^2)} \\ &= \left(\prod_{j=1}^{i+1} \frac{1}{1-tx_j^2}\right)\cdot t(x_i+x_{i+1}).\end{align*} The first expression is the generating function for $\partial_i({}^{\mathfrak{b}}h_{\ell}(1,i)) $, and the last expression is the generating function for  $(x_i+x_{i+1}){}^{\mathfrak{b}}h_{\ell}(1,i+1)$. The additional sign comes from the rule for taking the differential of a product in \eqref{dgstructure}. It follows that $\partial_i({}^{\mathfrak{b}}h_\ell(1,i)) =(x_i + x_{i+1}) {}^{\mathfrak{b}}h_{\ell-1}(1,i+1)$, as desired. 
\end{proof}

In type A, this DG-algebra structure provides a connection to the equivariant cohomology of Grassmannians via a quasi-isomorphism with cyclotomic quotients of the usual nilHecke algebra, cf. \cite[Proposition 8.3]{NV1} and \cite[Theorem 2.2]{AEHL}. The existence of differentials in the type B case may point to a similar geometric connection.

\section{Extended symmetric polynomials}
 In this section, we turn to investigate some of the combinatorial aspects of the extended type B nilHecke algebra and its action on extended polynomials. Following \cite{AEHL}, we define the extended symmetric polynomial ring: $${}^{\mathfrak{b}}\Lambda_n^{ext} := \bigcap_{i=1}^n \ker(\partial_i).$$ We will describe several bases of this algebra and demonstrate that $\nhb_n$ is a matrix algebra over it. 

\subsection{Low rank examples} Here we describe the structure of the extended symmetric polynomial ring for several low rank cases. These bases are computed directly by analyzing the action of each Demazure operator on a general extended polynomial, e.g. $\omega_1 + A\omega_2$, etc., and provide a framework for the general structure.

For $n>0$, let ${}^{\mathfrak{b}}\Lambda_n= \QQ[x_1,\ldots, x_n]^{W_{B_n}}$, the usual type B symmetric polynomials.

\noindent { $\underline{n=2}$}:  The algebra ${}^{\mathfrak{b}}\Lambda_2^{ext}$ is a free module of rank 4 over ${}^{\mathfrak{b}}\Lambda_2$ with basis $$\{1, \omega_1 + A \omega_2, \omega_2, \omega_1 \omega_2\}$$ where $A$ is a solution to the system $$\partial_1(A) = x_1+x_2, \quad\partial_2(A)=0.$$ For example, we could have $A= x_1^2 $ or $A=-(x_2^2 + x_3^2)$.

\noindent {$\underline{n=3}$}: The algebra ${}^{\mathfrak{b}}\Lambda_3^{ext}$ is a free module of rank 8 over ${}^{\mathfrak{b}}\Lambda_2$ with basis $$\begin{array}{cccc} 1 \quad &\omega_1 +A_1 \omega_2 + A_2\omega_3 \quad &\omega_2 + B \omega_3 \quad& \omega_3 \\ \omega_1 \omega_2 + C_1\omega_1 \omega_3 + C_2 \omega_2 \omega_3 \quad & \omega_1 \omega_3 + D\omega_2 \omega_3 \quad & \omega_2 \omega_3 \quad &\omega_1 \omega_2 \omega_3, \end{array},$$ where $A_i,B,C_i,D$ satisfy $$\partial_3(A_i) =0, \quad \partial_2(x)=0, \ \partial_2(A_2) =s_2(A_1) (x_2+x_3)\quad \partial_1(A_1) = x_1 +x_2, \ \partial_1(A_2) = 0; $$ $$\partial_1(B)=\partial_3(B)=0,\quad\partial_2(B)=x_2+x_3 ;$$ $$\partial_3(C_1) = \partial_3(C_2) = 0, \quad \partial_2(C_1) = x_2+x_3, \ \partial_2(C_2) = 0, \quad \partial_1(C_1)=0,\ \partial_1(C_2) = s_1(C_1)(x_1+x_2);$$ $$\partial_3(D) = \partial_2(D)=0, \quad \partial_1(D) = x_1 + x_2.$$ Hence, for example, the second basis element could take the form $\omega_1 + x_1^2\omega_2 + x_1^2x_2^2 \omega_3$, or $\omega_1 -(x_2^2+x_3^2) \omega_2 + x_3^4 \omega_3$. 

\subsection{Extended Schur polynomials}

The above bases can be realized as collections of extended Schur polynomials. Let $w_0\in W_{B_n}$ be the longest element; this takes the form $$(s_1 s_2 \ldots  s_n s_{n-1} \ldots s_1)(s_2 s_3 \ldots s_n s_{n-1} \ldots s_2)\ldots (s_{n-1} s_n s_{n-1}) s_n.$$ For any $w\in W_{B_n}$ with reduced expression $w= s_{i_1} \ldots s_{i_k},$ recall that $$\partial_w = \partial_{i_1}\ldots \partial_{i_k}.$$ 

For $\alpha= (\alpha_1, \ldots, \alpha_n)$ a partition (possibly with trailing zeros), denote by $$\underline{x}^{\delta+\alpha} = x_1^{2n-1+\alpha_1} x_2^{2n-3+\alpha_2} \ldots x_n^{\alpha_n}.$$ For a bounded strict partition $\beta$ of length $k\leq n$ with parts no larger than $n$, define $$\omega_\beta = \omega_{\beta_1} \omega_{\beta_2} \ldots \omega_{\beta_k}.$$ Finally, for any such partitions $\alpha$ and $\beta$, define the extended Schur polynomial $\mathcal{S}_{\alpha,\beta}$ by the following formula: $$\mathcal{S}_{\alpha,\beta} = \partial_{w_0}(\underline{x}^{\delta+\alpha} \omega_{\beta}).$$ Note that $\Sch_{\alpha,\beta} \in \ker \partial_i$ for all $i$, and hence is extended symmetric. 

The Schur polynomials with $\alpha=(0)$ and $\beta$ of length 1 have a regular structure in terms of homogeneous symmetric polynomials. 

\begin{lemma} We have $$\mathcal{S}_{0,i} = \sum_{\ell\geq i}(-1)^{n-i}\ {}^{\mathfrak{b}}\mathsf{h}_{\ell-i}(1,\ell-1)\omega_\ell,$$ where ${}^{\mathfrak{b}}\mathsf{h}_j(1,\ell-1)$ is the complete homogeneous symmetric polynomial in the variables $(x_1^2, x_{2}^2, \ldots, x_{\ell}^2)$. \end{lemma}
\begin{proof} We use an argument inspired by \cite[Lemma 2.13]{NV2}.

Write $\Sch_{0,i} = \sum_{\ell\geq i} c_\ell \omega_\ell$ for some $c_\ell \in \ZZ[x_1,\ldots,x_n].$ Since $\Sch_{0,i}$ is symmetric, we have $$0 = \partial_j\left(\sum_{\ell \geq i} c_\ell \omega_\ell\right) = - (s_j c_j)(x_j + x_{j+1})\omega_{j+1} + \sum_{\ell\geq i} (\partial_j(c_j)) \omega_j.$$ This implies that $$s_i(c_j)(x_j+x_{j+1})\omega_{j+1} = \partial_j(c_{j+1})\omega_{j+1},$$ and hence $\partial_j(c_{j+1}) = c_j(x_j +x_{j+1})$, since $s_i(\partial_j) = \partial_j$ for all $j$. Thus \begin{equation}\label{cjfromcn}\partial_j \partial_{j+1} \ldots \partial_{n-1} (c_n) = (x_{n-1}+x_n)(x_{n-2}+x_{n-1})\ldots(x_j+x_{j+1})c_j.\end{equation} Now, using the reduced expression $$\partial_{w_0} = (\partial_1 \partial_2\ldots\partial_{n-1} \partial_n \partial_{n-1}) \ldots \partial_1)(\partial_2\ldots \partial_n \partial_{n-1} \ldots \partial_2)\ldots \partial_n$$ we have $$\Sch_{0,i} = \sum_{k=i}^{n-1}c_k\omega_k +\ (-1)^{n-i}(\partial_1 \partial_2 \ldots \partial_{n-1} \partial_n \partial_{n-1} \ldots \partial_1)\ldots  \partial_n(\underline{x}^{\delta})) s_{n-1} \ldots s_i(\omega_{n}).$$ Note that this last coeffcient is just the usual (non-extended) type B Schur polynomial associated to the partition $(i)$. Hence we have $c_n = (-1)^{n-i} {}^{\mathfrak{b}}h_{n-i}(1,n-i).$ Now, using \eqref{cjfromcn} and \eqref{homgenpolys} gives $$c_j = (-1)^{n-i} h_{j-i}(x_j^2,\ldots, x_n^2)$$ as desired.  \end{proof}

This immediately implies a multiplication formula following the methods in \cite[Proposition 2.14]{NV2}.

\begin{lemma} For $\beta,\beta'$ strict partitions as above, we have $$\mathcal{S}_{0,\beta} \mathcal{S}_{0,\beta'} = (-1)^{\epsilon_{\beta,\beta'}}\mathcal{S}_{0,\beta \beta'},$$ where $\beta \beta'$ is the unique strict partition that can be formed from the set $\beta \sqcup \beta'$ and $\epsilon_{\beta, \beta'}$ is the length of the minimal permutation taking $\beta \sqcup \beta'$ to $\beta \beta'.$ \end{lemma}\begin{proof} First let $\beta=(i)$ and $\beta'=(j)$, with $i\not=j$. Then we have $\partial_{w_0}(\underline{x}^\delta\ {}^{\mathfrak{b}}h_{k-j}(1,k))=0$ unless $k=j$ and $\partial_{w_0}(\omega_i \omega_j)=0$ (since $\partial_{w_0}$ will eventually raise the smaller index to equal the larger one in each term). Thus \begin{align*} \Sch_{0,i}\Sch_{0,j} &= \partial_{w_0}(\underline{x}^\delta \omega_i \Sch_{0,j})\\&= \sum_{k\geq j} (-1)^{n-j} \partial_{w_0}(\underline{x}^\delta\ {}^{\mathfrak{b}}h_{k-j}(1,k) \omega_i \omega_j)\\&= \partial_{w_0}(\underline{x}^\delta \omega_i \omega_j) \\ &= (-1)^{\epsilon_{i,j}} \Sch_{0,ij}.\end{align*} Applying this reasoning to general strict partitions $\beta$ gives that $\mathcal{S}_{0,\beta} \mathcal{S}_{0,\ell} = (-1)^{\epsilon_{\beta,\ell}}(\mathcal{S}_{0,\beta\ell}).$ The result then follows from induction on the length of $\beta'.$ \end{proof}

We also have a basis for the extended symmetric polynomials. 

\begin{lemma}\label{basis} The abelian group $({}^{\mathfrak{b}}\Lambda_n^{ext})_{2k}$ consisting of all elements with $k$ total nonzero $\omega$'s has a ${}^{\mathfrak{b}}\Lambda_n$ basis given by $\{\mathcal{S}_{0,\nu}\}$, where $\nu$ ranges over strict partitions with $k$ parts. \end{lemma}
\begin{proof} This is a formal argument which carries over without modification from \cite[Proposition 2.15]{NV2}.\end{proof}

\subsection{Examples of Schur polynomials} 

We have the following low rank examples of the basis of Schur polynomials.\\

 \noindent $\underline{n=2}$: Let $\alpha= (0,0)$ and $\beta = \emptyset$. Then $\underline{x}^{\delta+\alpha} = x_1^3 x_2$, so we have \begin{align*} \mathcal{S}_{\alpha, \beta} &= \partial_1 \partial_2 \partial_1 \partial_2(x_1^3 x_2)= 1.\end{align*}

 Let $\beta =(1)$. Then we have \begin{align*} \mathcal{S}_{\alpha,\beta} &= \partial_2\partial_1\partial_2\partial_1(x_1^3x_2 \omega_1) = \omega_1 + x_1^2 \omega_2.\end{align*} Note in particular that this is one of the basis elements described in the previous section. 

 Let $\beta=(2)$. Then we have \begin{align*}  \mathcal{S}_{\alpha,\beta} = \partial_1 \partial_2 \partial_1 \partial_2( x_1^3 x_2 \omega_2) =  \omega_2. \end{align*} This is another of the basis elements described above.

 Finally, let $\beta=(1,2)$. Then we have \begin{align*} \mathcal{S}_{\alpha,\beta} = \partial_1 \partial_2 \partial_1 \partial_2( x_1^3 x_2 \omega_1\omega_2)= \omega_1 \omega_2. \end{align*} This completes the basis of ${}^{\mathfrak{b}}\Lambda_2^{ext}$ described in the previous section.

Note also that we have (cf. Lemma 4) $$\mathcal{S}_{0,(1)} \mathcal{S}_{0,(2)} = \mathcal{S}_{0,(1,2)},$$ $$\mathcal{S}_{0,(1)} \mathcal{S}_{0,(1,2)} = 0,$$ and $$\mathcal{S}_{0,(2)} \mathcal{S}_{0,(1,2)} = 0.$$ 

\noindent $\underline{n=3}$: Here we abbreviate the lengthy computations.
\begin{enumerate} \item $\mathcal{S}_{(0,0,0),\emptyset} = 1.$ \item $\mathcal{S}_{(0,0,0),(1)} = \omega_1 +x_1^2 \omega_2 + x_1^2x_2^2\omega_3.$ \end{enumerate}

In particular, we obtain a basis of the extended symmetric polynomials in this case as well.

\subsection{Extended nilHecke algebra as a matrix algebra}

Define, for $w\in W_{B_n}$, the type B Schubert polynomial $$\schub_w = \partial_{w^{-1}w_0}(\underline{x}^\delta),$$ with $w_0\in W_{B_n}$ and $\delta$ as above. The set $\{\schub_w\}_{w\in W_{B_n}}$ forms a basis for $\mathbb{Q}[x_1,\ldots,x_n]$ as a left $\mathbb{Q}[x_1,\ldots,x_n]^{W_{B_n}}$ module. Note that $\deg(\schub_w) = \ell(w)$ (cf \cite{FK}).

\begin{proposition} The set $\{\schub_w\}_{w\in W_{B_n}}$ forms a basis for $\extpoly$ as a free left module over $\extpoly^{B_n}$, the invariants of $\extpoly$ under the action of $W_{B_{n}}$. \end{proposition}

\begin{proof} We show that the multiplication map $$\extpoly^{B_n} \otimes \operatorname{span}_{\mathbb{Q}}\{\schub_w\}_{w\in W_{B_n}}\longrightarrow \extpoly$$ is an isomorphism. It is clearly injective. To show that it is surjective, we must write any polynomial $f\in\extpoly$ as a sum $f= \sum f_i b_i$, where $b_i \in \operatorname{span}_{\mathbb{Q}}\{\schub_w\}_{w\in W_{B_n}}$. 

Note that it suffices to focus on $f$ which contain a nonzero number of $\omega_i$'s, since the analogous result holds for usual polynomials. To see this, note that any such $f = p(\underline{x})\omega_\beta$ can be written as a linear combination of Schur polynomials $\mathcal{S}_{0,\beta}$. Indeed, since the term of $\mathcal{S}_{0,\beta}$ which contains the highest degree $\omega_i$ (with respect to the grading in Section 2.1, with $\deg(\omega)=-2i$) is $\omega_\beta$, we have $$p(\underline{x})\omega_\beta - p(\underline{x})\mathcal{S}_{0,\beta} = \sum_{\nu >
\beta} c_\nu \omega_\nu.$$ Repeatedly applying this procedure to the remaining $c_\nu \omega_\nu$ terms gives the desired decomposition. \end{proof}

Note that the graded rank of $\extpoly$ as a left $\extpoly^{B_n}$-module is thus $\sum_{w\in W_{B_n}} q^{\operatorname{deg} \schub_w} = \sum_{w\in W_{B_n}} q^{\ell(w)}$. Set $P_{B_n}(q) = \prod_{i=1}^n (1 + q + q^2 + \ldots + q^{2k})$; then, by \cite[Theorem 1.1]{Rei}, $\sum_{w\in W_{B_n}} q^{\ell(w)} = P_{B_{n}}(q).$ We therefore have the following isomorphism.

\begin{proposition} There is a $\ZZ$-algebra isomorphism $$\nhb \xrightarrow{\sim}\End_{{}^{\mathfrak{b}}\Lambda_n^{ext}}(\mathsf{P}_n^{ext})\cong \operatorname{Mat}_{P_{B_n}(q)} ({}^{\mathfrak{b}}\Lambda_n^{ext}).$$ \end{proposition}
\begin{proof} That this map is injective follows from the fact that $\nhb$ acts on $\mathsf{P}_n^{ext}$ via linearly independent operators. The surjectivity of the map follows from the freeness of $P_n^{ext}$ as a $\Lambda_n^{ext}$-module and a comparison of graded ranks. \end{proof}

\section{Extended Solomon's theorem for type B}
Fix an integer $n\geq 1$, and let $\mathbf{x} = \{x_1,x_2,\ldots, x_n\}$ and $\mathbf{dx} = \{dx_1, dx_2, \ldots dx_n\}$ be sets of formal even and odd variables, respectively. We use the following shorthand for the superpolynomials in $\mathbf{x}$ and $\mathbf{dx}$: $$\superpoly := \QQ[x_1,\ldots,x_n] \otimes \bigwedge[dx_1,\ldots dx_n].$$ This ring is bigraded with $\deg(x_i)=(1,0)$ and $\deg(dx_i)=(0,1)$. Note that there is an action of $W_{B_n}$ on $\bigwedge[dx_1, \ldots, dx_n]$ given be $s_i(dx_j) = dx_{s(j)}$ for $1\leq i \leq n-1$, $s_n(dx_j) = dx_j$ for $1\leq j \leq n-1$, and $s_n(dx_n) = -dx_n$.

 Solomon's theorem gives the following description of the $W$-invariants of this ring for any Weyl group $W$. 

\begin{theorem} \cite{Sol} For any family $\mathbf{f}=\{f_1,\ldots,f_n\}$ of algebraically independent generators of $\QQ[\mathbf{x}]^{W}$, $$\superpoly^W = \QQ[\mathbf{f},\mathbf{df}],$$ where for $f\in \QQ[\mathbf{x}]$, $$df = \sum_{i=1}^n \frac{\partial f}{\partial x_i} dx_i.$$ \end{theorem}

For $1\leq i \leq n-1$, let $\alpha_i =\frac{1}{x_i- x_{i+1}}$ and define $\mathbf{\alpha} = \{\alpha_1, \ldots,\alpha_{n-1}].$ In \cite{AEHL}, this theorem is enlarged to the extended polynomial ring $\extpoly$ by showing that there is an $\mathsf{NH}_n^{ext}$-equivariant isomorphism $\extpoly \rightarrow \QQ[\mathbf{x},\mathbf{dx}, \mathbf{\alpha}]$ which induces a canonical identification of $S_n$- invariants: $$\extpoly^{S_n} \xrightarrow{\sim} \QQ[\f,\df].$$ We aim to prove an analogous result for the type $B$ invariants.

It is first necessary to define an action of the type $B$ divided difference operators $\partial_i$ on $\superpoly.$ Define the denominator $\alpha_n = \frac{1}{2x_n}$ and set $\mathbf{\alpha} = \{\alpha_1,\ldots,\alpha_n\}$ and consider the algebra $\superpolydenoms.$ This algebra is bigraded, with $\operatorname{deg}(\alpha_i) = (-1,0)$. 

There is an action of the type $B_n$ divided difference operators on this algebra given by the usual action on polynomials and $$\partial_i(dx_i) = \frac{dx_i - s_i(dx_i)}{x_i - s_i(x_i)},$$ for $1 \leq i \leq n.$ Solomon's theorem implies that for any set $\mathbf{f} = \{f_1,\ldots, f_n\}$ of algebraically independent generators of $\QQ[\mathbf{x}]^{W_{B_n}}$, the subalgebra $\QQ[\mathbf{x},\mathbf{df}]$ is closed under the action of the divided difference operators.

We will construct a map $\superpoly \rightarrow \extpoly$ which will furnish an action of $\nhb$ on $\QQ[\mathbf{x},\mathbf{df}]$, and show that it is isomorphic as an $\nhb$-module to $\extpoly.$ 

As in \cite{AEHL}, call a tuple $\p\subset \QQ[\x]$ admissible if $p_j \in \QQ[\x]^{S_{n-1}}$, $s_3(p_j) = p_j$ 
, $\deg(p_j) = 2(n-j)$, and $\partial_{c[j]}(p_j)\in \QQ\setminus \{0\}$ for any $1\leq j \leq n$, where $$c[j] := s_{j+1} s_j s_{j+2} s_{j+1} \ldots s_n s_{n-1}$$ and $c[n]=1.$  We choose this last condition so that for an admissible tuple $\p$, the matrix $$\PP=(\partial_{c[j]}p_i)_{1\leq i,j \leq n}$$ is upper triangular and invertible, and so that $\PP$ contains the admissible tuple in its last column. 

\noindent {\bf Example 1.} Let $n=3$ and choose $\p = ((-1)^{3-i}h_{3-i}(x_3^2))_{i=1}^3.$ Certainly the $p_i$ have the correct degrees and are symmetric with respect to $S_2$ and $s_3$. Also, \begin{align*}\partial_{c[1]}(p_1) =\partial_2 \partial_1\partial_3\partial_2(x_3^4) = 1 \in \mathbb{Q}\setminus\{0\};\end{align*} \begin{align*} \partial_{c[2]}(p_2) = \partial_3\partial_2(-x_3^2) = 1\in \mathbb{Q}\setminus\{0\};\end{align*} and $\partial_{c[3]}(p_3) = 1.$ Hence $\p$ is admissible. The corresponding matrix is $$
\PP = \left(\begin{array}{ccc} 1 & -(x_2^2+x_3^2) & x_3^4 \\ 0 & 1 & -x_3^2 \\ 0&0&1 \end{array}\right).$$ Note in particular that $$\PP\bm{\omega}^T = \left(\begin{array}{c} \omega_1 - (x_2^2 + x_3^2)\omega_2 + x_3^4 \omega_3 \\ \omega_2- x_3^2\omega_3 \\ \omega_3\end{array} \right),$$  which are  degree 1 basis elements for $ {}^{\mathfrak{b}}\Lambda_3^{ext}$ as a ${}^{\mathfrak{b}}\Lambda_3$-module. 


For any $1\leq k \leq n-1$, define operators $\gamma_k, \rho_k: M_n(\extpoly) \rightarrow M_n(\extpoly)$ by setting $$\gamma_k(A)_{ij} = \delta_{j,k+1} A_{ik},$$ and $$\rho_k(A)_{ij} = \delta_{ik} A_{k+1, j}.$$ In other words, $\gamma_k$ returns the $k$th column of $A$ in the $k+1$st column, and $\rho_k$ returns the $k+1$ row in the $k$th row. We have the following characterization of admissible tuples.

\begin{lemma}\label{char1} \begin{enumerate} \item If $\p=\{p_1,\ldots, p_n\}$ is an admissible tuple, then $\PP$ satisfies $$\partial_{k+1}\partial_{k}(\PP) = \gamma_k(\PP)$$ for any $k=1,\ldots, n-1,$ and $\partial_n(\PP)=0$.
\item For any invertible $\mathsf{Q} = (q_{ij}) \in M_n(\QQ[x])$ such that $$\partial_{k+1}\partial_k(\mathsf{Q}) = \gamma_k(\mathsf{Q})$$ for $1\leq k \leq n-1$, $\partial_n(\QQ)=0$, and $\deg(q_{ij}) = 2(j-i)$, the tuple $\mathbf{q}=\{q_{1n},\ldots, q_{nn}\}$ is admissible and $\mathsf{Q}_{ij} = \partial_{c[j]} q_{in}$. 
\end{enumerate}\end{lemma}

\begin{proof} The first part follows from the fact that $p_i$ is symmetric with respect to $S_{n-1}$ and $s_n$, so that $$\partial_{k+1}\partial_k\partial_{c[j+1]} p_i = \delta_{j,k+1}\partial_{c[j]}p_i.$$ Indeed, this is obviously zero unless $j=k+1$, and, if $k<n-1$ \begin{align*}\partial_{k+1} \partial_k \partial_{c[k+1]}&= \partial_{k+1} \partial_k \partial_{k+2} \partial_{k+1} \ldots \partial_n \partial_{n-1}  \\ &= \partial_{c[k]};\end{align*} if $k=n-1$, then $\partial_n \partial_{n-1} \partial_{c[n]} = \partial_n \partial_{n-1} = \partial_{c[n-1]}.$ Clearly, $\partial_n(\PP)=0$.

The second part follows from similar calculations using the same proof as in type A, cf \cite[Lemma 4.3]{AEHL}\end{proof}

\noindent {\bf Example 2.} Let $n=3$ and $\p$ be as before. Then $$\partial_3\partial_2(\PP) = \left(\begin{array}{ccc}0&0& -(x_2^2+x_3^2) \\ 0&0&1\\0&0&0\end{array}\right).$$ This is clearly $\gamma_2(\PP)$. 

Also, $$\partial_2\partial_1(\PP) = \left(\begin{array}{ccc} 0&1&0 \\ 0&0&0 \\ 0&0 &0\end{array}\right) = \gamma_1(\PP).$$ \\


We have an additional characterization of matrices of admissible tuples which requires extra notation. Let $\Theta = \{\theta_i\}$ and $\Xi=\{\xi_i\}$ be two sets of algebraically independent elements in $\superpoly$ such that $\deg(\theta_i) = \deg(\xi_i) = 2(n-i)$, and define an invertible matrix $\mathsf{P} \in M_n(\QQ[x])$ by the relation $$\Xi = P \Theta.$$ We must therefore have $\deg(p_{ij}) = 2(j-i)$. 

\begin{lemma}\label{char2} Any two of these equations imply the third: 
\begin{enumerate} \item $\partial_{k+1}\partial_k(\PP) =  \gamma_k(\PP)$ for $1\leq k \leq n-1$, and $\partial_n(\PP)=0$.
\item $\partial_k(\Xi) = 0$ for $1\leq k \leq n.$
\item $\partial_k(\Theta) = -\rho_k(\Theta)$ for $1\leq k \leq n$.  
\end{enumerate}\end{lemma}

\begin{proof}We follow \cite[Lemma 4.4]{AEHL}. Let $k\leq n-1$ and suppose $\partial_{k+1}\partial_k(\PP) = \gamma_k(\PP)$. Note that we have $$s_{k+1}s_k(\gamma_{k+1} \gamma_k(\PP)) = s_{k+1}s_k(\partial_{k+1}\partial_k(\PP)) = \partial_{k+1}\partial_k(\PP) = \gamma_k(\PP),$$ and it follows that $ \gamma_k(\PP)\Theta = s_{k+1}s_k(\PP) s_{k+1}s_k(\Theta)=s_{k+1} s_k (\PP)\rho_{k+1}\rho_k(\Theta) .$ Now, using the definition of $\PP$, we have that if $\partial_k(\Xi)=0$, $\partial_{k+1}\partial_k(\Xi)=0$, so \begin{align*}0&=\partial_{k+1}\partial_k(\PP)\Theta + s_{k+1}(\partial_k(\PP)) \partial_{k+1}(\Theta) \\&\quad + \partial_{k+1}s_k(\PP) \partial_k(\Theta)+ s_{k+1}s_k(\PP)\partial_{k+1}\partial_k(\Theta)\\&= \gamma_k(\PP)\Theta+s_{k+1}(\partial_k(\PP)) \partial_{k+1}(\Theta) \\&\quad + \partial_{k+1}s_k(\PP) \partial_{k}(\Theta)+ s_{k+1}s_k(\PP)\partial_{k+1}\partial_k(\Theta) \\ &= s_{k+1} s_k(\PP)\rho_{k+1}\rho_k(\Theta) +s_{k+1}(\partial_k(\PP)) \partial_{k+1}(\Theta) \\&\quad + \partial_{k+1}s_k(\PP) \partial_{k}(\Theta)+ s_{k+1}s_k(\PP)\partial_{k+1}\partial_k(\Theta).\end{align*}

 Thus, $$s_{k+1}s_k(\PP)(\rho_{k+1}\rho_k(\Theta) +\partial_{k+1}\partial_k(\Theta))= -s_{k+1} \partial_k(\PP) \partial_{k+1}(\Theta) - \partial_{k} (s_k(\PP)) \partial_{k+1}(\Theta).$$

 Acting on both sides by $s_{k} s_{k+1}$ and using the identity $\partial_k s_k = -s_{k+1}\partial_k$ gives $$\PP(\rho_k(\Theta) +\partial_{k+1}\partial_k(\Theta)) = -s_k\partial_k(\PP)\partial_{k+1}(\Theta) + s_k\partial_k(\PP)\partial_{k+1}(\Theta) = 0.$$ 

Hence $\partial_k(\Xi)=0$ if and only if $\partial_{k+1}\partial_k(\Theta) = -\rho_k(\Theta)$, since $\PP$ is invertible (and the proof applies for $k=n$ because $s_n(\PP) = \PP$). 

The other equivalence follows using the same elementwise techniques as in type A. \end{proof}

Finally, we may construct our isomorphism. Let $\f = \{f_1,\ldots, f_n\}$ be a set of algebraically independent generators of $\QQ[\x]^{W_{B_n}}$, with $\deg(f_i) = 2(n-i)$. Let $\p = \{p_1,\ldots, p_n\}\subset \QQ[\x]$ be an admissible tuple and $\PP$ its associated matrix. 
\begin{theorem} For any choice of $\f$ and $\p$, there is a unique $\QQ[\x]$-linear homomorphism $${}^{\mathfrak{b}}\mathbf{\mathsf{J}}_p^f : \extpoly \rightarrow \superpolydenoms$$ defined by the relation $\df = \PP \J(\bm{\omega}).$ Further, $\J$ is injective and ${}^{\mathfrak{b}}NH_n$-equivariant. \end{theorem}

\begin{proof} We follow \cite[Proposition 4.5]{AEHL}. Since $\p$ is admissible, the matrix $\PP$ is invertible, and thus $\J$ is uniquely determined by this condition and linearity in $\QQ[\x]$. Injectivity of $\J$ follows from the invertibility of $\PP$ and the algebraic independence of the sets $\f$ and $\df$.

It remains to show that $\J$ respects the action of the divided difference operators. Since $\p$ is admissible, we have by Lemma \ref{char1} that $\partial_{k+1}\partial_k(\PP) =  \gamma_k(\PP)$. Note also that $\partial_k(\df)=0$, so by Lemma \ref{char2}, we must have $\partial_k(\J(\bm{\omega})) = -\rho_k(\J(\bm{\omega}).$ Then \begin{align*} \partial_k(\J(\omega_j)) &= \delta_{ij}\J(\omega_{j+1}). \end{align*} It follows that $$s_k(\J(\omega_j)) = \J(\omega_j) +a \delta_{jk}(x_{k} - x_{k+1}) \J(\omega_{j+1}) $$ for all $j,k<n$. Note that this matches the $W_{B_n}$ action on $\bm{\omega}$. Finally, if $k=n$, we have $$\partial_n(\J(\omega_j)) = \J(\omega_j),$$ since $\rho_k(\J(\omega_j)) = 0$. Hence $\J$ is $\nhb$-equivariant. 
\end{proof}

Hence $\J$ descends to a canonical identification of $W_{B_n}$-invariants $$\extpoly^{W_{B_n}} \cong \superpoly^{W_{B_n}}.$$

\end{document}